\documentclass[11pt]{amsart}
\usepackage[margin=1.15in]{geometry}
\usepackage{amscd,amssymb, amsmath, wasysym}
\usepackage{graphicx}
\usepackage{amsfonts}
\usepackage{mathrsfs}    
\usepackage{amsmath}    
\usepackage{amsthm}     
\usepackage{amscd}      
\usepackage{amssymb}    
\usepackage{eucal}      
\usepackage{latexsym}   
\usepackage{verbatim}   
\usepackage[all]{xy}     
\usepackage[dvipsnames]{xcolor}

\usepackage{bookmark}

 \usepackage{hyperref}
 \hypersetup{
     colorlinks=true,
     linkcolor=NavyBlue,
     filecolor=NavyBlue,
     citecolor = TealBlue,
     urlcolor=cyan,
  }

\newcounter{thmcounter}

\numberwithin{equation}{section}
\numberwithin{thmcounter}{section}

\newtheorem{theorem}[thmcounter]{Theorem}
\newtheorem{proposition}[thmcounter]{Proposition}
\newtheorem{lemma}[thmcounter]{Lemma}

\newtheorem{conjecture}[thmcounter]{Conjecture}
\newtheorem{problem}[thmcounter]{Problem}

\theoremstyle{definition}

\newtheorem{remark}[thmcounter]{Remark}

\newtheoremstyle{claim}{9pt}{3pt}{}{\parindent}{\bf}{.}{1em}{}

\theoremstyle{claim}

\newenvironment{namelist}[1]{%
\begin{list}{}
{
\settowidth{\labelwidth}{#1}%
\setlength{\labelsep}{0.3em}%
\setlength{\leftmargin}{\labelwidth}%
\addtolength{\leftmargin}{\labelsep}}}{%
\end{list}}

\newcommand{\nZ}{\mathbb{Z}}                     

\newcommand{\nP}{\mathbb{P}}                     

\newcommand{\sE}{\mathscr{E}}

\newcommand{\sO}{\mathscr{O}}                    

\newcommand{\mf}[1]{\mathfrak{#1}}

\DeclareMathOperator{\ev}{ev}					 

\DeclareMathOperator{\Tor}{Tor}                  

\DeclareMathOperator{\rank}{rank}                

\newcounter{rkcounter}             
\begin{document}

\title[Asymptotic vanishing of syzygies of algebraic varieties]{Asymptotic vanishing of syzygies of algebraic varieties}

\author{Jinhyung Park}
\address{Department of Mathematical Sciences, KAIST, 291 Daehak-ro, Yuseong-gu, Daejeon 34141, Republic of Korea}
\email{parkjh13@kaist.ac.kr}

\date{\today}
 
\thanks{J. Park was partially supported by the National Research Foundation (NRF) funded by the
Korea government (MSIT) (NRF-2021R1C1C1005479).}

\begin{abstract} 
The purpose of this paper is to prove Ein--Lazarsfeld's conjecture on asymptotic vanishing of syzygies of algebraic varieties. This result, together with Ein--Lazarsfeld's asymptotic nonvanishing theorem, describes the overall picture of asymptotic behaviors of the minimal free resolutions of the graded section rings of line bundles on a projective variety as the positivity of the line bundles grows. Previously, Raicu reduced the problem to the case of products of three projective spaces, and we resolve this case here.
\end{abstract}

\maketitle


\section{Introduction}

Throughout the paper, we work over an algebraically closed field $\Bbbk$ of arbitrary characteristic.
Let $X$ be a projective variety of dimension $n$, and $L$ be a very ample line bundle on $X$ which 
gives rise to an embedding
$$
X \subseteq \nP H^0(X, L) = \nP^{r},
$$
where $r=h^0(X, L)-1$.
Denote by $S$ the homogeneous coordinate ring of $\nP^{r}$. Fix a coherent sheaf $B$ on $X$, and let
$$
R=R(X, B; L):=\bigoplus_{m \in \mathbb{Z}} H^0(X, B \otimes L^m)
$$
be the graded section $S$-module of $B$ associated to $L$. By the Hilbert syzygy theorem, $R$ has a minimal free resolution
$$
 \xymatrix{
 0 & R \ar[l]& E_0 \ar[l]  & E_1  \ar[l] & \ar[l] \cdots &  \ar[l] E_r \ar[l]  &  \ar[l]0 
 }
$$
where
$$
E_p = \bigoplus_{q} K_{p,q}(X, B;L)  \otimes_{\Bbbk} S(-p-q).
$$
The \emph{Koszul cohomology} group $K_{p,q}(X, B;L)$ is the space of \emph{$p$-th syzygies of weight $q$}.  
When $B=\sO_X$, we set $K_{p,q}(X, L):=K_{p,q}(X, \sO_X; L)$.
After the pioneering work of Green \cite{Green1, Green2}, there has been a considerable amount of work to understand vanishing and nonvanishing of $K_{p,q}(X, B;L)$.

\medskip

We say that $L$ satisfies the \emph{property $N_k$} if $K_{0,1}(X, L)=0$ and $K_{p,q}(X, L)=0$ for $0 \leq p \leq k$ and $q \geq 2$. The property $N_0$ means that  $X \subseteq \nP^{r}$ is projectively normal, and the property $N_1$ means that the defining ideal of $X$ in $\nP^{r}$ is generated by quadratic polynomials. Thus the property $N_k$ provides a natural  framework to generalize classical results on defining equations of algebraic varieties to the results on their syzygies. Along this line, Green proved that if $X$ is a smooth projective complex curve of genus $g$ and $\deg L \geq 2g+1+k$, then $L$ satisfies the property $N_k$ (see \cite[Theorem (4.a.1)]{Green1}). Green's celebrated theorem has stimulated further work in this direction, and several analogous statements for higher dimensional algebraic varieties have been established, e.g. \cite{BCR1, BCR2, EL1, ENP,  GP, Green2, Pareschi}. On the other hand, Green--Lazarsfeld \cite{GL1, GL2} and Ottaviani--Paoletti \cite{OP} called attention to the failure of the property $N_k$. The main result of \cite{OP} asserts that
$$
K_{p,2}(\nP^2, \sO_{\nP^2}(d)) \neq 0~~\text{ for $3d-2 \leq p \leq r_d-2$},
$$
where $r_d = h^0(\nP^2, \sO_{\nP^2}(d))-1$. In particular, $\sO_{\nP^2}(d)$ does not satisfy the property $N_{3d-2}$. As $r_d \approx d^2/2$, the property $N_k$ for $\sO_{\nP^2}(d)$ describes only a small fraction of the syzygies of the $d$-th Veronese embedding of $\nP^2$. Eisenbud--Green--Hulek--Popescu observed in \cite[Proposition 3.4]{EGHP} that a similar phenomenon occurs for other smooth projective surfaces, and Ein--Lazarsfeld proved in \cite[Theorem A]{EL2} that this always happens for all smooth projective varieties.

\medskip

It is an interesting problem to describe the overall asymptotic behaviors of $K_{p,q}(X, B; L)$ as the positivity of $L$ grows. This type of question was first suggested by Green \cite[Problem 5.13]{Green1} and also considered by Ein--Lazarsfeld \cite[Problem 4.4]{EL1}. To set the stage for asymptotic syzygies of algebraic varieties, assume that $X$ is smooth and $B$ is a line bundle, and let
$$
L_d:=\sO_X(dA+P)~~\text{ for an integer $d \geq 1$,}
$$
where $A$ is an ample divisor and $P$ is an arbitrary divisor on $X$. We suppose that $d$ is sufficiently large, so in particular, $L_d$ is very ample. Put $r_d:=h^0(X, L_d)-1$. Elementary considerations of Castelnuovo--Mumford regularity show that 
$$
K_{p,q}(X, B; L_d)=0~~\text{ for $q \geq n+2$}.
$$
When $q=0$ or $n+1$, \cite[Proposition 5.1 and Corollary 5.2]{EL2} state that
$$
\begin{array}{l}
K_{p,0}(X, B; L_d) \neq 0 ~\Longleftrightarrow~0 \leq p \leq h^0(B)-1;\\
K_{p, n+1}(X, B; L_d) \neq 0 ~\Longleftrightarrow~r_d-n-h^0(X, K_X-B) + 1 \leq p \leq r_d-n.
\end{array}
$$
The main issue is then to study vanishing and nonvanishing of $K_{p,q}(X, B; L_d)$ for $1 \leq q \leq n$.
Motivated by the nonvanishing results of \cite{EGHP, OP}, Ein--Lazarsfeld established the \emph{asymptotic nonvanishing theorem} in \cite[Theorem 4.1]{EL2}: For each $1 \leq q \leq n$, there are constants $C_1, C_2 > 0$ such that if $d$ is sufficiently large, then
$$
K_{p,q}(X, B; L_d) \neq 0~~\text{ for $C_1 d^{q-1} \leq p \leq r_d - C_2 d^{n-1}$}.
$$
If furthermore $H^i(X, B) = 0 $ for $1 \leq i \leq n-1$, then
$$
K_{p,q}(X, B; L_d) \neq 0~~\text{ for $C_1 d^{q-1} \leq p \leq r_d - C_2 d^{n-q}$};
$$
in this case with $P=0$, a quick proof is provided in \cite{EEL2} with effective range of $p$. It is worth noting that Zhou adapted the arguments in \cite{EL2} to show that the asymptotic nonvanishing theorem also holds for singular projective varieties with a relaxed assumption on $B$ (see \cite[Remark 4.2]{EL2}). The influential paper \cite{EL2} opens the door to research on the asymptotic behaviors of $K_{p,q}(X, B; L_d)$ for $d$ increasing. We refer to \cite{EL4} for survey on recent progress on asymptotic syzygies of algebraic varieties.

\medskip

It is very natural to ask whether the $K_{p,q}(X, B; L_d)$ vanish for  the values of $p$ outside the range in the statement of the asymptotic nonvanishing theorem. Ein--Lazarsfeld conjectured that this is indeed the case for $p < O(d^{q-1})$ (see \cite[Conjecture 7.1]{EL2}, \cite[Conjecture 1.10]{EL4}). We confirm this \emph{asymptotic vanishing conjecture} in a more general setting. 
 
 \begin{theorem}\label{thm:asyvan}
Let $X$ be a projective variety of dimension $n$, and $B$ be a coherent sheaf on $X$. Fix an ample divisor $A$ and an arbitrary divisor $P$ on $X$, and put $L_d:=\sO_X(dA+P)$ for an integer $d \geq 1$.
For each $1 \leq q \leq n+1$, there is a constant $C>0$ depending on $X, A, B, P$ such that if $d$ is sufficiently large, then
$$
K_{p,q}(X, B; L_d) =0~~\text{ for $0 \leq p \leq Cd^{q-1}$}.
$$
\end{theorem}

We give some remarks on vanishing of $K_{p,q}(X, B; L_d)$ for large $p$. For simplicity, we assume that $X$ is smooth and $B$ is a vector bundle. If  $H^i(X, B)=0$ for $1 \leq i \leq n-1$, then the duality theorem (cf. \cite[Proposition 3.5]{EL2}, \cite[Theorem 2.c.6]{Green1}) says that 
$$
K_{p,q}(X, B; L_d) = K_{r_d-n-p, n+1-q}(X, B^* \otimes \omega_X; L_d)^*.
$$ 
Then our asymptotic vanishing theorem implies the following: For each $1 \leq q \leq n$, there is a constant $C>0$ such that if $d$ is sufficiently large, then
$$
K_{p,q}(X, B; L_d) =0~~\text{ for $p \geq r_d-Cd^{n-q}$}.
$$
However, if $H^{q-1}(X, B) \neq 0$ for some $2 \leq q \leq n$, then $K_{r_d-q+1, q}(X, B; L_d) \neq 0$ for large $d$ (see \cite[Remark 5.3]{EL2}). When $X$ is a smooth projective complex curve and $B$ is a line bundle, vanishing of weight-one syzygies $K_{p,1}(X, B; L_d)$ for large $p$ is determined by the duality theorem and \cite[Theorem B]{EL3} (see also \cite{Rathmann}). This implies Green--Lazarsfeld's gonality conjecture, and a higher dimensional generalization is treated in the work of Ein--Lazarsfeld--Yang \cite{ELY}.

\medskip

Shortly after the asymptotic vanishing conjecture was proposed, Raicu showed in the appendix of \cite{Raicu} that the general case of the conjecture follows from the case of products of three projective spaces. The case that $q=1$ in Theorem \ref{thm:asyvan} is trivial. To prove Theorem \ref{thm:asyvan}, it is more than enough to establish the following:

\begin{theorem}\label{thm:asyvanprodprojsp}
Let  $k \geq 1$ be an integer, $n_1, \ldots, n_k, d_1, \ldots, d_k$ be positive integers, and $b_1, \ldots, b_k$ be integers. Set
$$
\begin{array}{l}
X:=\nP^{n_1} \times \cdots \times \nP^{n_k}, ~B:=\sO_{\nP^{n_1}}(b_1) \boxtimes \cdots \boxtimes \sO_{\nP^{n_k}}(b_k),~L:=\sO_{\nP^{n_1}}(d_1) \boxtimes \cdots \boxtimes \sO_{\nP^{n_k}}(d_k),
\end{array}
$$
and $b:=\min\{b_1, \ldots, b_k\},~d:=\min\{d_1, \ldots, d_k\}$.
For each $2 \leq q \leq n_1 + \cdots +n_k+1$, if $d+b \geq 0$, then
$$
K_{p,q}(X, B;L) = 0~~\text{ for $0 \leq p \leq (1/n_1! \cdots n_k!)(d^{q-1}+bd^{q-2})$}.
$$
\end{theorem}

As $K_{p,q+1}(X, B; L) = K_{p,q}(X, B+L ; L)$, it is reasonable to assume that $b < d$ in Theorem \ref{thm:asyvanprodprojsp}. But we do not need this assumption for the proof.

\medskip

We give a sketch of the proof of Theorem \ref{thm:asyvanprodprojsp} for the Veronese case.
Let $M_{\sO_{\nP^n}(d)}$ be the kernel of the evaluation map $H^0(\nP^n, \sO_{\nP^n}(d)) \otimes \sO_{\nP^n} \to \sO_{\nP^n}(d)$. It is well known that
$$
K_{p,q}(\nP^n, \sO_{\nP^n}(d)) = H^{q-1}(\nP^n, \wedge^{p+q-1} M_{\sO_{\nP^n}(d)} (d))~~\text{ for $p \geq 0$ and $q \geq 2$}.
$$
The main idea is to work on $\nP^{n-1} \times \nP^1$ instead of $\nP^n$ via the finite map $\sigma \colon \nP^{n-1} \times \nP^1 \to \nP^n$ given by $(\xi, z) \mapsto \xi+z$, where $\nP^n$ is regarded as the Hilbert scheme of $n$ points on $\nP^1$ and $\sigma$ is the universal family. Note that $\sigma_*( \sO_{\nP^{n-1}} \boxtimes \sO_{\nP^1}(n-1)) = \sO_{\nP^n}^{\oplus n}$. For each $2 \leq q \leq n+1$, the problem is equivalent to showing that
$$
H^{q-1}(\nP^{n-1} \times \nP^1, \wedge^{p+q-1} \sigma^* M_{\sO_{\nP^n}(d)} \otimes \sO_{\nP^{n-1}}(d) \boxtimes \sO_{\nP^1}(d+n-1))=0~~\text{ for $0 \leq p \leq O(d^{q-1})$}.
$$
An advantage of working on $\nP^{n-1} \times \nP^1$ is that we can use the following short exact sequence
$$
0 \longrightarrow \bigoplus \sO_{\nP^{n-1}} \boxtimes \sO_{\nP^1}(-n) \longrightarrow  \sigma^* M_{\sO_{\nP^n}(d)} \longrightarrow  M_{\sO_{\nP^{n-1}}(d)} \boxtimes \sO_{\nP^1}(d) \longrightarrow  0,
$$
which provides a way to proceed by induction on $n$.
 By considering the natural filtration of $\wedge^{p+q-1}\sigma^*M_{\sO_{\nP^n}(d)}$, we reduce the problem to proving that
$$
H^{q-1}(\nP^{n-1} \times \nP^1, \wedge^i M_{\sO_{\nP^{n-1}}(d)} (d) \boxtimes \sO_{\nP^1}(a_i)) = 0~~\text{ for $0 \leq i \leq O(d^{q-1})$},
$$
where $a_i=id-(p+q-1-i)n+d+n-1$.
By induction on $n$, we can assume that $H^{j}(\nP^{n-1}, \wedge^i M_{\sO_{\nP^{n-1}}(d)} (d))=0$ for $0 \leq i \leq O(d^j)$ and $j=q-2, q-1$. By the K\"{u}nneth formula, it is sufficient to check that
$$
H^1(\nP^1, \sO_{\nP^1}(a_i))=0~~\text{ when $i \geq O(d^{q-2})$}.
$$
But we have 
$$
a_i = id + (d+2n+in-qn) - 1 - pn \geq O(d^{q-1})-1-pn \geq -1
$$
as soon as $0 \leq p \leq O(d^{q-1})$. Thus $K_{p,q}(\nP^n, \sO_{\nP^n}(d)) =0$ for this range of $p$.  The same argument works for the general Segre--Veronese case.

\medskip

There has been a great deal of attention to the syzygies of Veronese or Segre--Veronese embeddings, e.g. \cite{Bruce, BEGY, BCR1, BCR2, FX, ORS, OP, Raicu, Rubei}. The syzygies of these varieties have connections to representation theory and combinatorics. It would be exceedingly interesting to know whether the method of the present paper could make progress on the study of the Veronese or Segre--Veronese syzygies.

\medskip

The paper is organized as follows. After reviewing basic necessary facts in Section \ref{sec:prelim}, we prove Theorem \ref{thm:asyvanprodprojsp} in Section \ref{sec:asyvan}, where we also show Theorem \ref{thm:asyvan} following Raicu's argument in \cite{Raicu}. Section \ref{sec:openprob} is devoted to presenting some open problems on asymptotic syzygies of algebraic varieties.

\subsection*{Acknowledgements}
The author is very grateful to Lawrence Ein, Sijong Kwak, and Wenbo Niu for inspiring discussions and valuable comments, and he is indebted to Daniel Erman for introducing some references. The author would like to thank the referees for careful reading of the paper.

\section{Preliminaries}\label{sec:prelim}

We collect basic facts which are used to prove the main theorems of the paper.

\subsection{Koszul Cohomology}
Let $X$ be a projective variety, $B$ be a coherent sheaf on $X$, and $L$ be a very ample line bundle on $X$, which gives an embedding
$$
X \subseteq \nP H^0(X, L) = \nP^r.
$$
Let $S:=S(H^0(X, L)) = \bigoplus_{m \geq 0} S^mH^0(X, L)$ be the homogeneous coordinate ring of $\nP^r$, and
$$
R=R(X, B; L):=\bigoplus_{m \in \mathbb{Z}} H^0(X, B \otimes L^m)
$$
be the graded section $S$-module of $B$ associated to $L$. 
Denote by $S_+ \subseteq S$ the irrelevant maximal ideal, and define the \emph{Koszul cohomology group} to be
$$
K_{p,q}(X, B; L) := \Tor_p^S(R, S/S_+)_{p+q}.
$$
Then $R$ has a minimal free resolution
$$
 \xymatrix{
 0 & R \ar[l]& E_0 \ar[l]  & E_1  \ar[l] & \ar[l] \cdots & \ar[l] E_r \ar[l]  &  \ar[l]0 
 }
$$
where
$$
E_p = \bigoplus_{q} K_{p,q}(X, B; L)  \otimes_{\Bbbk} S(-p-q).
$$
Notice that $K_{p,q}(X, B; L)$ is the vector space of \emph{$p$-th syzygies of weight $q$} and it is  the cohomology of the Koszul-type complex
$$
\begin{array}{l}
\wedge^{p+1} H^0(X, L) \otimes H^0(X, B \otimes L^{q-1}) \longrightarrow \wedge^p  H^0(X, L)\otimes H^0(X, B \otimes L^q)\\[5pt]
~\text{ }~\text{ }~\text{ }~\text{ }~\text{ }~\text{ }~\text{ }~\text{ }~\text{ }~\text{ }~\text{ }~\text{ }~\text{ }~\text{ }~\text{ }~\text{ }~\text{ }~\text{ }~\text{ }~\text{ }~\text{ }~\text{ }~\text{ }~\text{ }~\text{ }~\text{ }~\text{ }~\text{ }~\text{ }~\text{ }~\text{ }
\longrightarrow \wedge^{p-1}H^0(X, L) \otimes H^0(X, B \otimes L^{q+1}).
\end{array}
$$
When $B=\sO_X$, we set $K_{p,q}(X, L) :=K_{p,q}(X, \sO_X; L)$.

\medskip

Now, let $L$ be a globally generated line bundle on a projective variety $X$.
Consider the evaluation map
$$
\ev \colon H^0(X, L)\otimes \sO_{X} \longrightarrow  L,
$$ 
which is surjective since $L$ is globally generated. Denote by $M_{L}$ the kernel bundle of the evaluation map $\ev$. Then we obtain a short exact sequence of vector bundles on $X$:
\begin{equation}\label{eq:M_Lses}
0\longrightarrow M_{L} \longrightarrow H^0(X, L)\otimes \sO_{X} \longrightarrow L \longrightarrow 0.
\end{equation}
We use the following well-known fact to compute the Koszul cohomology group.

\begin{proposition}[{cf. \cite[Corollary 3.3]{EL2}}]\label{prop:koszulcoh}
Let $X$ be a projective variety, $B$ be a coherent sheaf on $X$, and $L$ be a very ample line bundle on $X$.
Assume that $H^i(X, B \otimes L^m)=0$ for $i >0$ and $m >0$. Fix $q \geq 2$. Then we have
$$
K_{p,q}(X, B; L)=H^{q-1}(X, \wedge^{p+q-1}M_L \otimes B \otimes L)~\text{ for $p \geq 0$}.
$$
In particular, $K_{p,q}(X, B; L)=0$ for $0 \leq p \leq p_0$ if and only if 
$$
H^{q-1}(X, \wedge^j M_L \otimes B \otimes L) = 0~~\text{ for $0 \leq j \leq p_0 + q-1$}.
$$
\end{proposition}

\begin{proof}
By taking wedge product of (\ref{eq:M_Lses}), we have a short exact sequence
\begin{equation}\label{eq:wedgeM_Lses}
0 \longrightarrow \wedge^{p+1} M_L \longrightarrow \wedge^{p+1} H^0(X, L) \otimes \sO_X \longrightarrow  \wedge^{p} M_L \otimes L \longrightarrow  0.
\end{equation}
By using the Koszul-type complex and chasing through the diagram, we see that
$$
K_{p,q}(X, B; L) = H^1(X, \wedge^{p+1}M_L \otimes B \otimes L^{q-1})~~\text{ for $p \geq 0$ and $q \geq 2$}.
$$
See \cite[Section 2.1]{AN} or \cite[Section 1]{EL1} for the complete proof. Now, from (\ref{eq:wedgeM_Lses}), we find that
$$
H^1(X, \wedge^{p+1}M_L \otimes B \otimes L^{q-1}) = H^2(X, \wedge^{p+2}M_L \otimes B \otimes L^{q-2}) = \cdots = H^{q-1}(X, \wedge^{p+q-1}M_L \otimes B \otimes L), 
$$
so the first assertion holds. Thus $K_{p,q}(X, B; L)=0$ for $0 \leq p \leq p_0$ if and only if 
$$
H^{q-1}(X, \wedge^j M_L \otimes B \otimes L) = 0~~\text{ for $q-1 \leq j \leq p_0 + q-1$}.
$$
But we get from (\ref{eq:wedgeM_Lses}) that
$$
H^{q-1}(X, \wedge^i M_L \otimes B \otimes L) = H^{q-2}(X, \wedge^{i-1}M_L \otimes B \otimes L^2) = \cdots = H^{q-i-1}(X, B \otimes L^{i+1}) = 0
$$
for any $0 \leq i \leq q-2$. Thus the second assertion holds.
\end{proof}

We refer to \cite{AN, EL1, EL2, Eisenbud, Green1} for more details on syzygies and Koszul cohomology.

\subsection{Filtrations for Wedge Products}
For a short exact sequence $0 \to U \to V \to W \to 0$ of vector bundles on a projective variety $X$ and an integer $k \geq 1$, there is a natural filtration
\begin{equation}\label{eq:filt}
0 = F^0 \subseteq F^1 \subseteq \cdots \subseteq F^k \subseteq F^{k+1}=\wedge^k V
\end{equation}
such that 
$$
F^{p+1}/F^p = \wedge^{k-p}U \otimes \wedge^p W~~\text{ for all $0 \leq p \leq k$. }
$$

\begin{lemma}\label{lem:filt}
Let $0 \to U \to V \to W \to 0$ be a short exact sequence of vector bundles on a projective variety $X$.  If $H^q(X, \wedge^{k-p}U \otimes \wedge^p W)=0$ for all $0 \leq p \leq k$, then $H^q(X, \wedge^k V)=0$.
\end{lemma}

\begin{proof}
By considering the natural filtration (\ref{eq:filt}), we immediately obtain the lemma.
\end{proof}

\subsection{Divided and Symmetric Powers}
Let $V$ be a finite dimensional vector space over $\Bbbk$. For an integer $n \geq 1$, the symmetric group $\mf{S}_n$ naturally acts on the tensor power $T^n  V := V^{\otimes n}$ by permuting the factors.
The \emph{divided power} of $V$ is the subspace
$$
D^n V:=\{ \omega \in T^n V \mid \sigma(\omega) = \omega~\text{for all $\sigma \in \mf{S}_n$} \} \subseteq T^n V,
$$
while the \emph{symmetric power} $S^n V$ of $V$ is the quotient of $T^n V$ by the subspace spanned by $\sigma(\omega) - \omega$ for all $\omega \in T^n V$ and $\sigma \in \mf{S}_n$. When $n=0$, we set $T^0 V = D^0 V = S^0 V = \Bbbk$. We have a natural identification
$$
D^n V = (S^n V^*)^*.
$$
By composing the inclusion of $D^n V$ into $T^n V$ with the projection onto $S^n V$, we have a natural map $D^n V \to S^n V$. This map is an isomorphism in characteristic zero, but it may be neither injective nor surjective in general. We can also define divided and symmetric powers of vector bundles on projective varieties. We refer to \cite[Section 3]{AFPRW} for more details. 

\medskip

Let $0 \to U \to V \to W \to 0$ be a short exact sequence of vector bundles on a projective variety $X$ with $\rank W = 1$. Let $\pi \colon \nP(V^*) \to X$ be the natural projection. Note that $\sO_{\nP(V^*)}(-\nP(U^*)) = \sO_{\nP(V^*)}(-1) \otimes \pi^* W^*$. For $k \geq 0$, we have a short exact sequence on $\nP(V^*)$:
$$
0 \longrightarrow \sO_{\nP(V^*)}(k) \otimes  \pi^* W^* \longrightarrow \sO_{\nP(V^*)}(k+1) \longrightarrow \sO_{\nP(U^*)}(k+1) \longrightarrow 0.
$$
By applying $\pi_*$, we get a short exact sequence on $X$:
$$
0 \longrightarrow S^k V^* \otimes W^* \longrightarrow S^{k+1} V^* \longrightarrow S^{k+1} U^* \longrightarrow 0.
$$
This construction was suggested by Lawrence Ein, and a purely algebraic construction of this kind of exact sequence can be found in \cite[Corollary V.1.15]{ABW}.
By taking the dual, we obtain a short exact sequence on $X$:
\begin{equation}\label{eq:divpowses}
0 \longrightarrow D^{k+1} U \longrightarrow D^{k+1} V \longrightarrow D^k V \otimes W \longrightarrow 0.
\end{equation}
We remark that $S^{k+1} V \to S^k V \otimes W$ may not be surjective in positive characteristic.

\medskip

Now, let $C$ be a smooth projective curve, and $L$ be a line bundle on $C$. For an integer $k \geq 0$, the symmetric group $\mf{S}_{k+1}$ naturally acts on the $(k+1)$-th ordinary product $C^{k+1}$ of $C$ by permuting the components, and the line bundle 
$$
L^{\boxtimes k+1}:=\underbrace{L \boxtimes \cdots \boxtimes L}_{\text{$k+1$ times}}
$$
on  $C^{k+1}$ descends to a line bundle $T_{k+1}(L)$ on the $(k+1)$-th symmetric product $C_{k+1}$ of $C$ (see \cite[Subsection 3.1]{ENP}). Note that
$$
H^0(C_{k+1}, T_{k+1}(L)) = H^0(C^{k+1}, L^{\boxtimes k+1})^{\mf{S}_{k+1}} = D^{k+1} H^0(C, L).
$$
If $C=\nP^1$ and $L = \sO_{\nP^1}(d)$ with $d \geq 1$, then $(\nP^1)_n = \nP^n$ and $T_n(\sO_{\nP^1}(d)) = \sO_{\nP^n}(d)$. By an arbitrary characteristic version of Hermite reciprocity (see \cite[Remark 3.2]{AFPRW}), we have
$$
\begin{array}{l}
H^0(\nP^n, T_n(\sO_{\nP^1}(d))) ~=~D^n H^0(\nP^1, \sO_{\nP^1}(d)) ~=~D^n (S^d H^0(\nP^1, \sO_{\nP^1}(1))) \\[5pt]
~~ = ~S^d (D^n H^0(\nP^1, \sO_{\nP^1}(1))) ~=~S^d H^0(\nP^n, T_n(\sO_{\nP^1}(1))).
\end{array}
$$

\subsection{Tautological Bundles on Projective Spaces}
Recall that $\sigma \colon \nP^{n-1} \times \nP^1 \to \nP^n$ is the finite map of degree $n$ given by $(\xi, z) \mapsto \xi+z$ by viewing $\nP^n$ as the Hilbert scheme of $n$ points on $\nP^1$. For any integer $k$, the \emph{tautological bundle} on $\nP^n$ is defined as
$$
E_{n, \sO_{\nP^1}(k)} := \sigma_* (\sO_{\nP^{n-1}} \boxtimes \sO_{\nP^1}(k)),
$$ 
which is a vector bundle of rank $n$. The tautological bundles on symmetric products of curves play an important role in the study of secant varieties of curves (see \cite{ENP}).

\begin{lemma}\label{lem:tautobund}
$E_{n, \sO_{\nP^1}(k)}$ is splitting if and only if $-1 \leq k \leq n-1$. In this case, 
$$
E_{n, \sO_{\nP^1}(k)} = \sO_{\nP^n}^{\oplus k+1} \oplus \sO_{\nP^n}(-1)^{\oplus n-1-k}.
$$
\end{lemma}

\begin{proof}
Since $\sigma \colon \nP^{n-1} \times \nP^1 \to \nP^n$ is a finite map, we have 
$$
H^i(\nP^n, E_{n, \sO_{\nP^1}(k)}(m)) = H^i(\nP^{n-1} \times \nP^1, \sO_{\nP^{n-1}}(m) \boxtimes \sO_{\nP^1}(m+k)).
$$
for any $i \geq 0$ and $m \in \nZ$.
The K\"{u}nneth formula shows that
\begin{small}
$$
\begin{array}{l}
H^i(\nP^{n-1} \times \nP^1, \sO_{\nP^{n-1}}(m) \boxtimes \sO_{\nP^1}(m+k)) \\[5pt]
= \big(H^{i-1}(\nP^{n-1}, \sO_{\nP^{n-1}}(m) ) \otimes H^1(\nP^1, \sO_{\nP^1}(m+k))\big) \oplus \big(H^{i}(\nP^{n-1}, \sO_{\nP^{n-1}}(m) ) \otimes H^0(\nP^1, \sO_{\nP^1}(m+k))\big).
\end{array}
$$
\end{small}\\[-7pt]
Then we see that
$$
\begin{array}{l}
H^i(\nP^n, E_{n, \sO_{\nP^1}(k)}(m)) = 0~~\text{ for $2 \leq i \leq n-2$ and $m \in \nZ$},\\[5pt]
H^1(\nP^n, E_{n, \sO_{\nP^1}(k)}(m)) = 0~\text{for $m \in \nZ$} ~ \Longleftrightarrow  ~k \geq -1, \\[5pt]
H^{n-1}(\nP^n, E_{n, \sO_{\nP^1}(k)}(m))=0~\text{for $m \in \nZ$} ~ \Longleftrightarrow ~k \leq n-1.
\end{array}
$$
By the Horrocks criterion, the first assertion of the lemma follows. Now, we observe that $h^0(\nP^n, E_{n, \sO_{\nP^1}(k)})=k+1$, and $h^0(\nP^n,  E_{n, \sO_{\nP^1}(k)}(1)) = n(k+2) = (n+1)(k+1) + n-1-k$. This implies the second assertion of the lemma.
\end{proof}

\begin{remark}
The following alternative approach to Lemma \ref{lem:tautobund} was suggested by Lawrence Ein. Let $D_n$ be the image of the injective map $\nP^{n-1} \times \nP^1 \to \nP^n \times \nP^1$ given by $(\xi, z) \mapsto (\xi+z, z)$. Note that $\sO_{\nP^n \times \nP^1}(-D_n) = \sO_{\nP^n}(-1) \boxtimes \sO_{\nP^1}(-n)$. For any integer $k$, we have a short exact sequence on $\nP^n \times \nP^1$:
$$
0 \longrightarrow \sO_{\nP^n}(-1) \boxtimes \sO_{\nP^1}(k-n) \longrightarrow \sO_{\nP^n} \boxtimes \sO_{\nP^1}(k) \longrightarrow \sO_{\nP^{n-1}} \boxtimes \sO_{\nP^1}(k) \longrightarrow 0.
$$
Let $p \colon \nP^n \times \nP^1 \to \nP^n$ be the projection to the first component.
When $-1 \leq k \leq n-1$, by applying $p_*$, we obtain a short exact sequence on $\nP^n$:
$$
0 \longrightarrow H^0(\nP^1, \sO_{\nP^1}(k)) \otimes \sO_{\nP^n} \longrightarrow E_{n, \sO_{\nP^1}(k)} \longrightarrow H^1(\nP^1, \sO_{\nP^1}(k-n)) \otimes \sO_{\nP^n}(-1) \longrightarrow 0.
$$
Thus $E_{n, \sO_{\nP^1}(k)}  = (H^0(\nP^1, \sO_{\nP^1}(k)) \otimes \sO_{\nP^n}) \oplus (H^1(\nP^1, \sO_{\nP^1}(k-n)) \otimes \sO_{\nP^n}(-1))$. When $k \leq -2$ or $k \geq n$, it is easy to check that $E_{n, \sO_{\nP^1}(k)}$ is not splitting.
\end{remark}

\begin{lemma}\label{lem:tautobundvan}
Let $Y$ be a projective variety, and $\sigma \colon Y \times \nP^{n-1} \times \nP^1 \to Y \times \nP^n$ be the finite map given by $(y, \xi, z) \mapsto (y, \xi+z)$. If $M$ is a vector bundle on $Y \times \nP^n$, then 
$$
H^q(Y \times \nP^n, M)=0~\Longleftrightarrow~H^q(Y \times \nP^{n-1} \times \nP^1, \sigma^* M \otimes (\sO_{Y \times \nP^{n-1}} \boxtimes \sO_{\nP^1}(n-1)))  = 0
$$
for any $q \geq 0$.
\end{lemma}

\begin{proof}
By Lemma \ref{lem:tautobund} and the projection formula,
$\sigma_{*}\big(\sigma^* M \otimes (\sO_{Y \times \nP^{n-1}} \boxtimes \sO_{\nP^1}(n-1))\big) = M^{\oplus n}$.
Since $\sigma$ is a finite map, the assertion immediately follows.
\end{proof}

\section{Asymptotic Vanishing Theorem}\label{sec:asyvan}

The aim of this section is to prove Theorem \ref{thm:asyvan}. First, we construct a short exact sequence of vector bundles, which allows us to give a quick proof of Theorem \ref{thm:asyvanprodprojsp} by induction on dimension. We then explain how one can deduce Theorem \ref{thm:asyvan} from Theorem \ref{thm:asyvanprodprojsp}.

\subsection{Short Exact Sequence}
Let $C$ be a smooth projective curve, and $L_C$ be a line bundle on $C$. Let $Y$ be a smooth projective variety, and $L_Y$ be a line bundle on $Y$. Fix an integer $k \geq 0$. Assume that the line bundle $L_Y \boxtimes T_{k+1}(L_C)$ on $Y \times C_{k+1}$ is globally generated. Notice that
$$
H^0( Y \times C_{k+1}, L_Y \boxtimes T_{k+1}(L_C)) = H^0(Y, L_Y) \otimes D^{k+1} H^0(C, L_C).
$$
We have a short exact sequence on $Y \times C_{k+1}$:
$$
0 \longrightarrow  M_{L_Y \boxtimes T_{k+1}(L_C)} \longrightarrow  H^0(Y, L_Y) \otimes D^{k+1}H^0(C, L_C) \otimes \sO_{Y \times C_{k+1}} \longrightarrow  L_Y \boxtimes T_{k+1}(L_C) \longrightarrow  0.
$$
We can view $C_{k+1}=\{\text{effective divisors of degree $k+1$ on $C$}\}$ as the Hilbert scheme of $k+1$ points on $C$.
Let 
$$
\sigma \colon Y \times C_k \times C \longrightarrow Y \times C_{k+1}
$$ 
be the finite morphism given by $(y, \xi, z) \mapsto (y, \xi + z)$, and 
$p \colon Y \times C_k \times C \rightarrow C$
be the projection to the last component. 
By taking $\sigma^*$ of the above exact sequence on $Y \times C_{k+1}$, we get a short exact sequence on $Y \times C_k \times C$:
$$
0 \to \sigma^* M_{L_Y \boxtimes T_{k+1}(L_C)} \to H^0(Y, L_Y) \otimes D^{k+1}H^0(C, L_C) \otimes \sO_{Y \times C_{k} \times C} \to L_Y \boxtimes T_{k}(L_C) \boxtimes L_C \to 0,
$$
By taking $p_*$ and considering (\ref{eq:divpowses}), we get a short exact sequence on $C$:
\begin{footnotesize}
$$
0 \to H^0(Y, L_Y) \otimes D^{k+1}M_{L_C} \to H^0(Y, L_Y) \otimes D^{k+1}H^0(C, L_C) \otimes \sO_C \to H^0(Y, L_Y) \otimes D^k H^0(C, L_C) \otimes L_C \to 0 .
$$
\end{footnotesize}\\[-14pt]
Then we obtain the following commutative diagram with exact sequences on $Y \times C_k \times C$:

\begin{tiny}
$$
\xymatrixrowsep{0.27in}
\xymatrixcolsep{0.09in}
\xymatrix{
&&&0 \ar[d] &\\
&0 \ar[d] &&M_{L_Y \boxtimes T_k(L_C)} \boxtimes L_C \ar[d] &\\
0 \ar[r] & p^*(H^0(Y, L_Y) \otimes D^{k+1}M_{L_C}) \ar[r] \ar[d] & p^*(H^0(Y, L_Y) \otimes D^{k+1}H^0(C, L_C) \otimes \sO_{C}) \ar[r] \ar@{=}[d] & p^*(H^0(Y, L_Y) \otimes D^k H^0(C, L_C) \otimes L_C )\ar[r] \ar[d] & 0\\
0 \ar[r] & \sigma^* M_{L_Y \boxtimes T_{k+1}(L_C)} \ar[r] \ar[d] & H^0(Y, L_Y) \otimes D^{k+1} H^0(C, L_C) \otimes \sO_{Y \times C_k \times C} \ar[r] & L_Y \boxtimes T_{k}(L_C) \boxtimes L_C \ar[r] \ar[d] & 0\\
&M_{L_Y \boxtimes T_k(L_C)} \boxtimes L_C \ar[d] && 0 & \\
&0 &&  & 
}$$
\end{tiny}

\medskip

Now, assume that $C=\nP^1$ and $L_C = \sO_{\nP^1}(d)$ with $d \geq 1$. For an integer $n \geq 1$, we have
$$
(\nP^1)_n = \nP^n,~~T_n(\sO_{\nP^1}(d)) = \sO_{\nP^n}(d),~~M_{L_Y \boxtimes T_n(\sO_{\nP^1}(d))} = M_{L_Y \boxtimes \sO_{\nP^n}(d)}.
$$
Since $M_{\sO_{\nP^1}(d)} = \bigoplus \sO_{\nP^1}(-1)$, it follows that
$$
D^{n} M_{\sO_{\nP^1}(d)} = \big(S^{n} M_{\sO_{\nP^1}(d)}^* \big)^* = \bigoplus \sO_{\nP^1}(-n).
$$
Then the left vertical short exact sequence in the above commutative diagram gives a short exact sequence on $Y \times \nP^{n-1} \times \nP^1$:
\begin{equation}\label{eq:mainses}
0 \longrightarrow  \bigoplus \sO_{Y \times \nP^{n-1}} \boxtimes \sO_{\nP^1}(-n) \longrightarrow  \sigma^* M_{L_Y \boxtimes \sO_{\nP^n}(d)} \longrightarrow  M_{L_Y \boxtimes \sO_{\nP^{n-1}}(d)} \boxtimes \sO_{\nP^1}(d) \longrightarrow  0.
\end{equation}
When $Y$ is a point, the exact sequence (\ref{eq:mainses}) is
$$
0 \longrightarrow  \bigoplus \sO_{\nP^{n-1}} \boxtimes \sO_{\nP^1}(-n) \longrightarrow  \sigma^* M_{\sO_{\nP^n}(d)} \longrightarrow  M_{\sO_{\nP^{n-1}}(d)} \boxtimes \sO_{\nP^1}(d) \longrightarrow  0.
$$
When $n=1$, the finite map $\sigma$ is an isomorphism and the exact sequence (\ref{eq:mainses}) is
$$
0 \longrightarrow  \bigoplus \sO_Y \boxtimes \sO_{\nP^1}(-1) \longrightarrow  M_{L_Y \boxtimes \sO_{\nP^1}(d)} \longrightarrow  M_{L_Y} \boxtimes \sO_{\nP^1}(d) \longrightarrow  0.
$$

\subsection{Case of Product of Projective Spaces}
In this subsection, we prove Theorem \ref{thm:asyvanprodprojsp}. Recall that
$$
\begin{array}{l}
X=\nP^{n_1} \times \cdots \times \nP^{n_k}, ~B=\sO_{\nP^{n_1}}(b_1) \boxtimes \cdots \boxtimes \sO_{\nP^{n_k}}(b_k),~L=\sO_{\nP^{n_1}}(d_1) \boxtimes \cdots \boxtimes \sO_{\nP^{n_k}}(d_k),
\end{array}
$$
and $b=\min\{b_1, \ldots, b_k\},~d=\min\{d_1, \ldots, d_k\}$. Fix $2 \leq q \leq n_1 + \cdots +n_k+1$. Our aim is to show that if $d +b \geq 0$, then
$$
K_{p,q}(X, B;L) = 0~~\text{ for $0 \leq p \leq (1/n_1! \cdots n_k!)(d^{q-1}+bd^{q-2})$.}
$$

\medskip

We put
$$
Y:=\nP^{n_1} \times \cdots \times \nP^{n_{k-1}},~B_Y:=\sO_{\nP^{n_1}}(b_1) \boxtimes \cdots \boxtimes \sO_{\nP^{n_{k-1}}}(b_{k-1}),~L_Y:=\sO_{\nP^{n_1}}(d_1) \boxtimes \cdots \boxtimes \sO_{\nP^{n_{k-1}}}(d_{k-1}),
$$
and $n:=n_k$. Then 
$$
X=Y \times \nP^n,~B=B_Y \boxtimes \sO_{\nP^n}(b_k),~L=L_Y \boxtimes \sO_{\nP^n}(d_k).
$$
As $md+b \geq 0$ for any $m>0$, we have
$$
H^i(X, B \otimes L^m)=0~~\text{ for $i>0$ and $m > 0$}.
$$
By Proposition \ref{prop:koszulcoh}, for $p \geq 0$, we have
$$
K_{p,q}(X, B;L) = H^{q-1}\big(Y \times \nP^n, \wedge^{p+q-1}M_{L_Y \boxtimes \sO_{\nP^n}(d_k)} \otimes ((L_Y+B_Y) \boxtimes \sO_{\nP^n}(d_k+b_k))\big).
$$

\medskip

We proceed by induction on $n_1 + \cdots + n_k$. If $n_1+ \cdots + n_k = 1$, then $q=2$ and the problem is to check the cohomology vanishing
$$
H^1\big(\nP^1, \wedge^{p+1} M_{\sO_{\nP^1}(d)} \otimes \sO_{\nP^1}(d+b)\big) = 0~~\text{ for $0 \leq p \leq d+b$.}
$$
 As $\wedge^{p+1}M_{\sO_{\nP^1}(d)} \otimes \sO_{\nP^1}(d+b) = \bigoplus \sO_{\nP^1}(d+b-p-1)$ and $d+b-p-1 \geq -1$, the desired cohomology vanishing immediately follows. 

\medskip

Assume that $n_1 + \cdots + n_k \geq 2$. Fix $0 \leq p \leq (1/n_1! \cdots n_k!)(d^{q-1}+bd^{q-2})$. By Lemma \ref{lem:tautobundvan}, it is sufficient to show the cohomology vanishing on $Y \times \nP^{n-1} \times \nP^1$:
$$
H^{q-1}\big(\wedge^{p+q-1}\sigma^* M_{L_Y \boxtimes \sO_{\nP^n}(d_k)} \otimes ((L_Y+B_Y) \boxtimes \sO_{\nP^{n-1}}(d_k+b_k) \boxtimes \sO_{\nP^1}(d_k+b_k+n-1))\big)=0,
$$
where $\sigma \colon Y \times \nP^{n-1} \times \nP^1 \to Y \times \nP^n$ is the finite map given by $(y, \xi, z) \mapsto (y, \xi+z)$.
By considering the short exact sequence (\ref{eq:mainses}) and applying Lemma \ref{lem:filt} to $\wedge^{p+q-1}\sigma^* M_{L_Y \boxtimes \sO_{\nP^n}(d_k)}$, we can reduce the problem to proving the following:
$$
H^{q-1}\big(Y \times \nP^{n-1} \times \nP^1, (\wedge^i M_{L_Y \boxtimes \sO_{\nP^{n-1}}(d_k)} \otimes ((L_Y + B_Y) \boxtimes \sO_{\nP^{n-1}}(d_k+b_k))) \boxtimes \sO_{\nP^1}(a_i ) \big)= 0
$$
for $0 \leq i \leq p+q-1~ (\leq (1/n_1! \cdots n_k!)(d^{q-1}+bd^{q-2}) + q-1)$, where
$$
a_i:=id_k - (p+q-1-i)n + d_k+b_k+n-1 = i(d_k+n) + d_k + b_k +2n-qn -1 -pn.
$$
By the K\"{u}nneth formula, it is equivalent to showing that
\begin{equation}\label{eq:cohvan1}
H^{q-1}\big(Y \times \nP^{n-1}, \wedge^i M_{L_Y \boxtimes \sO_{\nP^{n-1}}(d_k)} \otimes ((L_Y + B_Y) \boxtimes \sO_{\nP^{n-1}}(d_k+b_k))\big) \otimes H^0(\nP^1,  \sO_{\nP^1}(a_i)) = 0
\end{equation}
and
\begin{equation}\label{eq:cohvan2}
H^{q-2}\big(Y \times \nP^{n-1}, \wedge^i M_{L_Y \boxtimes \sO_{\nP^{n-1}}(d_k)} \otimes ((L_Y + B_Y)) \boxtimes \sO_{\nP^{n-1}}(d_k+b_k)\big) \otimes H^1(\nP^1,  \sO_{\nP^1}(a_i)) = 0
\end{equation}
for $0 \leq i \leq p+q-1$. 

\medskip

By induction and Proposition \ref{prop:koszulcoh}, we can assume that
\begin{equation}\label{eq:cohvanind1}
H^{q-1}\big(Y \times \nP^{n-1}, \wedge^i M_{L_Y \boxtimes \sO_{\nP^{n-1}}(d_k)} \otimes ((L_Y + B_Y) \boxtimes \sO_{\nP^{n-1}}(d_k+b_k))\big) = 0
\end{equation}
for $0 \leq i \leq (1/n_1! \cdots n_{k-1}!(n_k-1)!)(d^{q-1}+bd^{q-2})+q-1$ when $q \geq 2$, and
\begin{equation}\label{eq:cohvanind2}
H^{q-2}\big(Y \times \nP^{n-1}, \wedge^i M_{L_Y \boxtimes \sO_{\nP^{n-1}}(d_k)} \otimes ((L_Y + B_Y) \boxtimes \sO_{\nP^{n-1}}(d_k+b_k))\big) = 0
\end{equation}
for $0 \leq i \leq (1/n_1! \cdots n_{k-1}!(n_k-1)!)(d^{q-2}+bd^{q-3})+q-2$ when $q \geq 3$. 
Observe that (\ref{eq:cohvan1}) immediately follows from (\ref{eq:cohvanind1}). It only remains to check (\ref{eq:cohvan2}). If $q=2$, then
$$
a_i = i(d_k+n) + d_k+b_k-1 - np \geq d_k+b_k-1 -np \geq d+b-1-np \geq -1
$$
since $p \leq (1/n_1! \cdots n_k!)(d+b) \leq (1/n)(d+b)$. Thus
$$
H^1(\nP^1,  \sO_{\nP^1}(a_i)) = 0,
$$
so the cohomology vanishing (\ref{eq:cohvan2}) holds for $q=2$. Next, we consider the case that $q \geq 3$. If $0 \leq i \leq (1/n_1! \cdots n_{k-1}!(n_k-1)!)(d^{q-2}+bd^{q-3})+q-2$, then (\ref{eq:cohvan2}) immediately follows from  (\ref{eq:cohvanind2}). If $i \geq (1/n_1! \cdots n_{k-1}!(n_k-1)!)(d^{q-2}+bd^{q-3})+q-1$, then
$$
\begin{array}{rcl}
a_i &=&id_k + \underbrace{(d_k+b_k + 2n +in-qn)}_{\geq 0} -1- np \\
&\geq &id - 1-np\\
&\geq &(1/n_1! \cdots n_{k-1}!(n-1)!)(d^{q-1}+bd^{q-2})-1- np\\
&\geq & -1
\end{array}
$$
since $p \leq (1/n_1! \cdots n_{k-1}!n!)(d^{q-1}+bd^{q-2})$. Thus
$$
H^1(\nP^1,  \sO_{\nP^1}(a_i) ) = 0,
$$
so the cohomology vanishing (\ref{eq:cohvan2}) holds in this case as well. We have shown (\ref{eq:cohvan1}) and (\ref{eq:cohvan2}), and they imply $K_{p,q}(X, B;L)=0$ as desired. We complete the proof of Theorem \ref{thm:asyvanprodprojsp}.

\subsection{General Case}

In this subsection, we prove Theorem \ref{thm:asyvan}. As we mentioned in the introduction, Raicu proved that Theorem \ref{thm:asyvan} can be deduced from Theorem \ref{thm:asyvanprodprojsp} for the case $k=3$ (see \cite[Corollary A.5]{Raicu}). Here we reproduce his proof for the completeness.
Recall that $X$ is an $n$-dimensional projective variety, $B$ is coherent sheaf, $A$ is an ample divisor, $P$ is an arbitrary divisor on $X$, and $L_d:=\sO_X(dA+P)$ for an integer $d \geq 1$.
Our aim is to show that for each $2 \leq q \leq n+1$ (the case that $q=1$ is trivial), there is a constant $C>0$ depending on $X, A, B, P$ such that if $d$ is sufficiently large, then
$$
K_{p,q}(X, B; L_d) =0~~\text{ for $0 \leq p \leq Cd^{q-1}$}.
$$

\medskip

We can choose integers $a_1, a_2, a_3 \geq 1$ with $\text{gcd}(a_1, a_2+a_3)=1$ such that 
$$
A_1:=a_1A, ~A_2:=a_2A + P, ~A_3:=a_3A-P
$$ 
are very ample and the natural maps
\begin{equation}\label{eq:projnorm}
S^{m_1}H^0(X, A_1) \otimes S^{m_2}H^0(X, A_2) \otimes S^{m_3}H^0(X, A_3) \longrightarrow H^0(X, m_1A_1 + m_2 A_2 + m_3 A_3)
\end{equation}
are surjective for all $m_1, m_2, m_3 > 0$. We may assume that $a_1 \gg a_2+ a_3$. Note that $a_1, a_2, a_3$ are depending only on $X, A, P$.
As $d$ is sufficiently large, we can find integers $d_1, d_2 \geq 1$ such that $d_1 \approx d/2a_1, d_2 \approx d/2(a_2+a_3)$ and $d+a_3 = a_1d_1 + (a_2+a_3)d_2$. Let $d_3:=d_2-1$. Then $d=a_1d_1 + a_2d_2 + a_3d_3$, and $L_d=\sO_X(d_1A_1 + d_2A_2 + d_3A_3)$. Note that $d_1 < d_3 < d_2$.

\medskip

Next, consider the commutative diagrams
$$
\xymatrix{
X \ar@{^{(}->}[r] \ar@{^{(}->}[d] & \nP H^0(X, A_1) \times \nP H^0(X, A_2) \times \nP H^0(X, A_3)=:\nP^{n_1} \times \nP^{n_2} \times \nP^{n_3}=Y \ar@{^{(}->}[d]\\
\nP^r:=\nP H^0(X, L_d) \ar@{^{(}->}[r] & \nP H^0(Y, \sO_{\nP^{n_1}}(d_1) \boxtimes \sO_{\nP^{n_2}}(d_2) \boxtimes \sO_{\nP^{n_3}}(d_3)) =: \nP^N.
}
$$
Clearly, $n_1, n_2, n_3$ are depending only on $X, A, P$.
Notice that $\nP^r$ is a linear subspace of $\nP^N$ by the surjectivity of (\ref{eq:projnorm}). We can regard $B$ as a coherent sheaf on $Y, ~\nP^r,$ and $\nP^N$. The syzygies of $B$ on $\nP^N$ are the syzygies of $B$ on $\nP^r$ tensoring with a Koszul complex of linear forms.
By letting 
$$
L_Y:=\sO_{\nP^{n_1}}(d_1) \boxtimes \sO_{\nP^{n_2}}(d_2) \boxtimes \sO_{\nP^{n_3}}(d_3),
$$ 
we see that
\begin{equation}\label{eq:compKpq}
\min\{ p \mid K_{p,q}(X, B; L_d) \neq 0 \} = \min\{ p \mid K_{p,q}(Y, B; L_Y) \neq 0\}.
\end{equation}

\medskip

Let $S:=S(H^0(X, A_1) \oplus H^0(X, A_2) \oplus H^0(X, A_3))$ be the total coordinate ring of $Y=\nP^{n_1} \times \nP^{n_2} \times \nP^{n_3}$ with the usual $\mathbb{Z}^3$-grading. Then
$$
M:=\bigoplus_{(a_1, a_2, a_3) \in \mathbb{Z}^3_{\geq 0}} H^0(Y, B \otimes \sO_{\nP^{n_1}}(a_1) \boxtimes \sO_{\nP^{n_2}}(a_2) \boxtimes \sO_{\nP^{n_3}}(a_3))
$$
is a finitely generated graded $S$-module. Consider the minimal free resolution of $M$:
$$
 \xymatrix{
 0 & M \ar[l]& E_0 \ar[l]  & E_1  \ar[l] & \ar[l] \cdots &  \ar[l] E_m \ar[l]  &  \ar[l]0 
 }
$$
where
$$
E_i=\bigoplus_{b^j=(b_1^j, b_2^j, b_3^j) \in S_i} F_{i,b^j} \otimes S(b^j)
$$
for some finite dimensional vector space $F_{i, b^j}$ over $\Bbbk$ and finite subsets $S_i \subseteq \mathbb{Z}^3$. Let 
$$
B_Y^j:=\sO_{\nP^{n_1}}(b_1^j) \boxtimes \sO_{\nP^{n_2}}(b_2^j) \boxtimes \sO_{\nP^{n_3}}(b_3^j)~~\text{ and }~~ b:=\min\{ b_1^j, b_2^j, b_3^j \mid b^j \in S_i, 0 \leq i \leq m \}.
$$
Note that $b$ is depending only on $X, A, B, P$.
Now, fix $2 \leq q \leq n+1 ~(\leq n_1+n_2 +n_3+1)$. By Theorem \ref{thm:asyvanprodprojsp}, for $0 \leq i \leq m$ and $b^j \in S_i$, we have
\begin{equation}\label{eq:K_{p-i,q+i}=0}
K_{p-i,q+i}(Y, B_Y^j; L_Y) = 0~~\text{ for $0 \leq p \leq (1/n_1!n_2!n_3!)(d_1^{q+i-1}+bd_1^{q+i-2})+i$}.
\end{equation}
Then (\ref{eq:compKpq}) and \cite[Theorem A.1]{Raicu} (cf. \cite[Proposition 1.d.3]{Green1}) show that
\begin{equation}\label{eq:K_{p,q}=0}
K_{p,q}(X, B;L_d)=K_{p,q}(Y, B; L_Y)=0~~\text{ for $0 \leq p \leq (1/n_1!n_2!n_3!)(d_1^{q-1}+bd_1^{q-2})$}.
\end{equation}
Recall that the numbers $a_1, a_2, a_3, n_1, n_2, n_3, b$ are depending only on $X, A, B, P$ but not on $d$. Since $d_1 \approx d/2a_1$ grows linearly in $d$ and $d$ is sufficiently large, Theorem \ref{thm:asyvan} follows.

\begin{remark}
Instead of applying \cite[Theorem A.1]{Raicu}, one can alternatively prove (\ref{eq:K_{p,q}=0}) as follows. From the minimal free resolution of $M$, we have an exact sequence on $Y$:
$$
 \xymatrix{
 0 & B \ar[l]& \sE_0 \ar[l]  & \sE_1  \ar[l] & \ar[l] \cdots &  \ar[l] \sE_m \ar[l]  &  \ar[l]0 
 }
$$
where
$$
\sE_i=\bigoplus_{b^j \in S_i} F_{i, b^j} \otimes B_Y^j.
$$
By using Proposition \ref{prop:koszulcoh} and chasing through the above exact sequence, we see that
$$
K_{p,q}(Y, B; L_Y) = H^{q-1}(Y, \wedge^{p+q-1}M_{L_Y} \otimes B \otimes L_Y)=0
$$
is implied by 
$$
K_{p-i,q+i}(Y, B_Y^j; L_Y) = H^{q+i-1}(Y, \wedge^{(p-i)+(q+i)-1}M_{L_Y} \otimes B_Y^j \otimes L_Y)=0
$$
for all $0 \leq i \leq m$ and $b^j \in S_i$. Thus (\ref{eq:compKpq}) and (\ref{eq:K_{p-i,q+i}=0}) imply (\ref{eq:K_{p,q}=0}).
\end{remark}

\section{Open Problems}\label{sec:openprob}

In this section, we discuss some open problems and conjectures. Let $X$ be a smooth projective variety of dimension $n$, and $B$ be a coherent sheaf on $X$. Fix an ample divisor $A$ and an arbitrary divisor $P$ on $X$, and put $L_d:=\sO_X(dA+P)$ for an integer $d \geq 1$. 

\medskip

For each $2 \leq q \leq n+1$, it would be extremely interesting to find an explicit constant $c>0$ in terms of $X, A, B, P$, and $q, d$ such that if $d$ is sufficiently large, then
$$
\text{$K_{p,q}(X, B; L_d)=0$ for $0 \leq p \leq c$ and $K_{c+1, q}(X, B; L_d) \neq 0$.}
$$
However, this problem is already very difficult for $q=2$. A generalization of Mukai's conjecture (cf. \cite[Conjecture 4.2]{EL1}) asks whether the property $N_d$ holds for $K_X+(n+2+d)A$ when $X$ is a smooth projective complex variety. But it is widely open even when $n=2$ and $d=0$. Moreover, Fujita's conjecture, which predicts that $K_X + (n+2+d)A$ is very ample for $d \geq 0$, is unknown when $n \geq 3$. However, when $A$ is very ample, Ein--Lazarsfeld established in \cite[Theorem 1]{EL1} that
$$
K_{p,q}(X, K_X + (n+1+d)A) = 0~~\text{ for $0 \leq p \leq d$ and $q \geq 2$}.
$$
It is reasonable to expect extending this result for $q \geq 3$.

\begin{problem}
Let $X$ be a smooth projective complex variety of dimension $n$, and $A$ be a very ample divisor on $X$. 
For each $2 \leq q \leq n+1$ and $d \geq 0$, find an explicit polynomial $P(x)$ of degree $q-1$ such that
$$
K_{p,q}(X, K_X + (n+1+d)A) = 0~~\text{ for $0 \leq p \leq P(d)$}
$$
\end{problem}

One can also consider the effective asymptotic vanishing problem for the syzygies of products of projective spaces.

\begin{problem}\label{prob:prodprojsp}
Let  $k \geq 1$ be an integer, $n_1, \ldots, n_k, d_1, \ldots, d_k$ be positive integers, and $b_1, \ldots, b_k$ be integers. Set
$$
\begin{array}{l}
X:=\nP^{n_1} \times \cdots \times \nP^{n_k}, ~B:=\sO_{\nP^{n_1}}(b_1) \boxtimes \cdots \boxtimes \sO_{\nP^{n_k}}(b_k),~L:=\sO_{\nP^{n_1}}(d_1) \boxtimes \cdots \boxtimes \sO_{\nP^{n_k}}(d_k),
\end{array}
$$
For each $2 \leq q \leq n_1 + \cdots +n_k+1$, find a constant $c>0$ in terms of $b_1, \ldots, b_k, d_1, \ldots, d_k$, and $q$ such that if $d_1, \ldots, d_k$ are sufficiently large, then
\begin{equation}\label{eq:probprodprojsp}
\text{$K_{p,q}(X, B;L) = 0$ for $0 \leq p \leq c$ and $K_{c+1, q}(X, B;L) \neq 0$.}
\end{equation}
\end{problem}

In Theorem \ref{thm:asyvanprodprojsp}, we prove that if $d+b \geq 0$, then
$$
K_{p,q}(X, B;L) = 0~~\text{ for $0 \leq p \leq (1/n_1! \cdots n_k!)(d^{q-1}+bd^{q-2})$},
$$
where $b:=\min\{b_1, \ldots, b_k\}, ~d:=\min\{d_1, \ldots, d_k\}$.  By a small modification of the proof, we can improve the bound on $p$, but our method only gives a bound on $p$ depending on dimensions $n_1, \ldots, n_k$. It is expected that the constant $c$ in Problem \ref{prob:prodprojsp} is independent of $n_1, \ldots, n_k$. It would be also interesting to find a constant $d>0$ in terms of $b_1, \ldots, b_k, n_1, \ldots, n_k$, and $q$ such that if $d_1, \ldots, d_k \geq d$, then (\ref{eq:probprodprojsp}) holds. When $k=1$, there is a precise prediction on $c$ and $d$ (see \cite[Conjecture 7.5]{EL2}, \cite[Conjecture 2.3]{EL4}).

\begin{conjecture}[Ein--Lazarsfeld]\label{conj:ELvan}
Fix $n \geq 1$, $b \geq 0$, and $0 \leq q \leq n$. If $d \geq b+q+1$, then
$$
K_{p,q}(\nP^n, \sO_{\nP^n}(b); \sO_{\nP^n}(d)) = 0~~\text{ for }~0 \leq p \leq {d+q \choose q} - {d - b- 1 \choose q} - q-1.
$$
\end{conjecture}

Notice that the conjecture gives the precise vanishing range because Ein--Erman--Lazarsfeld \cite[Theorem 2.1]{EEL2} (see also  \cite[Theorem 2.1]{EL4}) proved that $K_{p,q}(\nP^n, \sO_{\nP^n}(b); \sO_{\nP^n}(d)) \neq 0$ for all 
$$
{d+q \choose q} - {d - b- 1 \choose q} - q
 \leq p \leq {d+n \choose n} + {d+n-q \choose n-q} - {n+b \choose q+b} - q -1,
 $$
 
\medskip 
 
In \cite{OP}, Ottaviani--Paoletti conjectured that if $n \geq 3, d \geq 3$, then $\sO_{\nP^n}(d)$ satisfies the property $N_{3d-3}$. They also consider the cases that $n \leq 2$ or $d \leq 2$, but these cases are already settled. By \cite[Theorem 1.6]{OP}, the property $N_{3d-3}$ for $\sO_{\nP^n}(d)$ is implied by that $K_{p,2}(\nP^n, \sO_{\nP^n}(d)) = 0$ for $0 \leq p \leq 3d-3$. Thus Conjecture \ref{conj:ELvan} for $b=0$ and $q=2$ is equivalent to Ottaviani--Paoletti's conjecture. At this moment, we only know that $\sO_{\nP^n}(d)$ satisfies the property $N_{d+1}$ by Bruns--Conca--R\"{o}mer \cite{BCR1}, and a small change of the proof of Theorem \ref{thm:asyvanprodprojsp} yields that $\sO_{\nP^3}(d)$ satisfies the property $N_{d+2}$. A new idea might be needed to solve Conjecture \ref{conj:ELvan}.

\medskip

It is also a fascinating problem to study the asymptotic behavior of the \emph{Betti numbers}
$$
k_{p,q}(X, B;L_d) := \dim K_{p,q}(X, B; L_d)
$$ 
when $d$ is sufficiently large (see \cite[Problem 7.3]{EL2}). In this direction, Ein--Erman--Lazarsfeld conjectured that for each $1 \leq q \leq n$, the Betti numbers $k_{p,q}(X, L_d)$ converge to a normal distribution (see \cite[Conjecture B]{EEL1}, \cite[Conjecture 3.2]{EL4}). This normal distribution conjecture has not been verified even for $\nP^2$ and $\nP^1 \times \nP^1$, and it seems that the conjecture is already very challenging for Veronese embeddings (cf. \cite{BCEGLY, BEGY}).

\begin{conjecture}[Ein--Erman--Lazarsfeld]\label{conj:normdist}
Fix $n \geq 1$ and $1 \leq q \leq n$. Then there is a normalizing function $F_q(d)$ such that
$$
F_q(d) \cdot k_{p_d, q}(\nP^n, \sO_{\nP^n}(d)) \longrightarrow e^{-a^2/2}
$$
as $d \to \infty$ and $p_d \to r_d/2 + a\sqrt{r_d}/2$.
\end{conjecture}

Notice that 
$$
k_{p, q}(\nP^n, \sO_{\nP^n}(d)) = \frac{1}{n} \cdot h^{q}\big(\nP^{n-1} \times \nP^1,  \wedge^{p+q}\sigma^* M_{\sO_{\nP^n}(d)} \otimes (\sO_{\nP^{n-1}} \boxtimes \sO_{\nP^1}(n-1)) \big)
$$
for $p \geq 0$ and $1 \leq q \leq n$. It is tempting to wonder if there is a clever way to compute this Betti number.

\medskip

We refer to \cite{BEGY} and \cite{EL2, EL4} for more problems and conjectures on syzygies of Veronese embeddings and asymptotic syzygies of algebraic varieties, respectively.

\bibliographystyle{ams}

\end{document}